\UseRawInputEncoding

\documentclass[oneside, 11pt]{amsart} 

\usepackage{amsmath,amsthm,amssymb,epic}
\usepackage{eqlist,eqparbox}
\usepackage{amsfonts}
\usepackage{latexsym}
\usepackage[2emode]{psfrag}
\usepackage{amsthm}
\usepackage{amsmath}
\usepackage[all]{xy}

\newcommand{\Title}[1]{\bigskip\bigskip\centerline{\bf #1}\bigskip}
\newcommand{\Author}[1]{\medskip\centerline{ \it #1}}

\newcommand{\Affiliation}[1]{\medskip\centerline{#1}}
\newcommand{\Email}[1]{\medskip\centerline{#1}\bigskip}

\begin{document}

\newcommand{\N}{\mbox {$\mathbb N $}}
\newcommand{\Z}{\mbox {$\mathbb Z $}}
\newcommand{\Q}{\mbox {$\mathbb Q $}}
\newcommand{\R}{\mbox {$\mathbb R $}}
\newcommand{\lo }{\longrightarrow }
\newcommand{\ul}{\underleftarrow }
\newcommand{\rl}{\underrightarrow }
\newcommand{\rs }{\rightsquigarrow }
\newcommand{\ra }{\rightarrow }
\newcommand{\dd }{\rightsquigarrow }
\newcommand{\ol }{\overline }
\newcommand{\la }{\langle }
\newcommand{\tr }{\triangle }
\newcommand{\xr }{\xrightarrow }
\newcommand{\de }{\delta }
\newcommand{\pa }{\partial }
\newcommand{\LR }{\Longleftrightarrow }
\newcommand{\Ri }{\Rightarrow }
\newcommand{\va }{\varphi }
\newcommand{\Den}{{\rm Den}\,}
\newcommand{\Ker}{{\rm Ker}\,}
\newcommand{\Reg}{{\rm Reg}\,}
\newcommand{\Fix}{{\rm Fix}\,}
\newcommand{\Sup}{{\rm Sup}\,}
\newcommand{\Inf}{{\rm Inf}\,}
\newcommand{\Img}{{\rm Im}\,}
\newcommand{\Id}{{\rm Id}\,}
\newcommand{\ord}{{\rm ord}\,}

\newtheorem{theorem}{Theorem}[section]
\newtheorem{lemma}[theorem]{Lemma}
\newtheorem{proposition}[theorem]{Proposition}
\newtheorem{corollary}[theorem]{Corollary}
\newtheorem{definition}[theorem]{Definition}
\newtheorem{example}[theorem]{Example}
\newtheorem{examples}[theorem]{Examples}
\newtheorem{xca}[theorem]{Exercise}
\theoremstyle{remark}
\newtheorem{remark}[theorem]{Remark}
\newtheorem{remarks}[theorem]{Remarks}
\numberwithin{equation}{section}

\def\leftmark{L.C. Ciungu}

\Title{CENTERS OF QUANTUM-WAJSBERG ALGEBRAS} 
\title[Centers of quantum-Wajsberg algebras]{}
                                                                           
\Author{\textbf{LAVINIA CORINA CIUNGU}}
\Affiliation{Department of Mathematics} 
\Affiliation{St Francis College}
\Affiliation{179 Livingston Street, Brooklyn, New York, NY 11201, USA}
\Email{lciungu@sfc.edu}

\begin{abstract} 
We define the Wajsberg-center and the OML-center of a quantum-Wajsberg algebra, and study their structures. 
We prove that the Wajsberg-center is a Wajsberg subalgebra of a quantum-Wajsberg algebra, and that it is a 
distributive sublattice of its corresponding poset.   
If the quantum-Wajsberg algebra is quasi-linear, we show that the Wajsberg-center is a linearly ordered Wajsberg algebra. We also show that the lattice subreduct of the Wajsberg-center is a Kleene algebra. 
Furthermore, we prove that the OML-center is an orthomodular lattice, and that the orthomodular lattices form a 
subvariety of the variety of quantum-Wajsberg algebras. \\

\textbf{AMS classification (2020):} 06F35, 03G25, 06A06, 81P10, 06C15
\end{abstract}

\maketitle

\section{Introduction} 

In the last decades, the study of algebraic structures related to the logical foundations of quantum mechanics 
became a central topic of research. Generally known as quantum structures, these algebras serve as algebraic 
semantics for the classical and non-classical logics, as well as for the quantum logics. 
As algebraic structures connected with quantum logics we mention the following algebras: bounded involutive lattices,  De Morgan algebras, ortholattices, orthomodular lattices, MV algebras, quantum MV algebras. \\
\indent
The quantum-MV algebras (or QMV algebras) were introduced by R. Giuntini in \cite{Giunt1} 
as non-lattice generalizations of MV algebras (\cite{Chang}) and as non-idempotent generalizations of orthomodular lattices (\cite{Beran, Kalm1}). 
These structures were intensively studied by R. Giuntini (\cite{Giunt2, Giunt3, Giunt4, Giunt5, Giunt6}), 
A. Dvure\v censkij and S. Pulmannov\'a (\cite{DvPu}), R. Giuntini and S. Pulmannov\'a (\cite{Giunt7}) and by 
A. Iorgulescu in \cite{Ior30, Ior31, Ior32, Ior33, Ior34, Ior35}. 
An extensive study on the orthomodular structures as quantum logics can be found in \cite{Ptak}. 
Many algebraic semantics for the classical and non-classical logics studied so far (pseudo-effect algebras, residuated lattices, pseudo-MV/BL/MTL algebras, bounded non-commutative R$\ell$-monoids, pseudo-hoops, pseudo-BCK/BCI algebras), as well as their commutative versions, are quantum-B algebras. \\
\indent
Quantum-B algebras, defined and investigated by W. Rump and Y.C. Yang (\cite{Rump2, Rump1}), 
arise from the concept of quantales which was introduced in 1984 as a framework for quantum mechanics 
with a view toward non-commutative logic (\cite{Mulv1}).
Interesting results on quantum-B algebras have been presented in \cite{Rump3, Rump4, Han1, Han2}. \\
\indent
We redefined in \cite{Ciu78} the quantum-MV algebras starting from involutive BE algebras and we introduced and 
studied the notion of quantum-Wajsberg algebras (QW algebras, for short). 
We proved that any Wajsberg algebra is a quantum-Wajsberg algebra, and the commutative quantum-Wajsberg algebras 
are Wajsberg algebras. 
It was also shown that the Wajsberg algebras are both quantum-Wajsberg algebras and commutative quantum-B algebras. \\
\indent
In this paper, we define the Wajsberg-center or the commutative center of a quantum-Wajsberg algebra $X$ as the set of 
those  elements of $X$ that commute with all other elements of $X$.  
We study certain properties of the Wajsberg-center, and we prove that the Wajsberg-center is a Wajsberg 
subalgebra of $X$, and it is also a distributive sublattice of its corresponding poset.   
If the quantum-Wajsberg algebra is quasi-linear, we show that the Wajsberg-center is a linearly ordered 
Wajsberg algebra. We also prove that the lattice subreduct of the Wajsberg-center is a Kleene algebra.  
Furthermore, we define the OML-center of a quantum-Wajsberg algebra, and study its properties. 
We prove that the OML-center is an orthomodular lattice, and that the orthomodular lattices form a subvariety of the 
variety of quantum-Wajsberg algebras.
Additionally, we prove new properties of quantum-Wajsberg algebras.

$\vspace*{1mm}$

\section{Preliminaries}

In this section, we recall some basic notions and results regarding BCK algebras, Wajsberg algebras, BE algebras and quantum-Wajsberg algebras that will be used in the paper. Additionally, we prove new properties of quantum-Wajsberg algebras. For more details regarding the quantum-Wajsberg algebras we refer the reader to \cite{Ciu78}. \\
\indent
Starting from the systems of positive implicational calculus, weak systems of positive implicational calculus 
and BCI and BCK systems, in 1966 Y. Imai and K. Is\`eki introduced the \emph{BCK algebras} (\cite{Imai}). 
BCK algebras are also used in a dual form, with an implication $\ra$ and with one constant element $1$, 
that is the greatest element (\cite{Kim2}). 
A (dual) BCK algebra is an algebra $(X,\ra,1)$ of type $(2,0)$ satisfying the following conditions, 
for all $x,y,z\in X$: 
$(BCK_1)$ $(x\ra y)\ra ((y\ra z)\ra (x\ra z))=1;$ 
$(BCK_2)$ $1\ra x=x;$ 
$(BCK_3)$ $x\ra 1=1;$ 
$(BCK_4)$ $x\ra y=1$ and $y\ra x=1$ imply $x=y$. 
In this paper, we use the dual BCK algebras. 
If $(X,\ra,1)$ is a BCK algebra, for $x,y\in X$ we define the relation $\le$ by $x\le y$ if and only if $x\ra y=1$, 
and $\le$ is a partial order on $X$. \\
\indent
\emph{Wajsberg algebras} were introduced in 1984 by Font, Rodriguez and Torrens in \cite{Font1} as algebraic model 
of $\aleph_0$-valued \L ukasiewicz logic.   
A \emph{Wajsberg algebra} is an algebra $(X,\ra,^*,1)$ of type $(2,1,0)$ satisfying the following conditions 
for all $x,y,z\in X$: 
$(W_1)$ $1\ra x=x;$ 
$(W_2)$ $(y\ra z)\ra ((z\ra x)\ra (y\ra x))=1;$ 
$(W_3)$ $(x\ra y)\ra y=(y\ra x)\ra x;$ 
$(W_4)$ $(x^*\ra y^*)\ra (y\ra x)=1$. 
Wajsberg algebras are bounded with $0=1^*$, and they are involutive. 
It was proved in \cite{Font1} that Wajsberg algebras are termwise equivalent to MV algebras. \\
\indent
\emph{BE algebras} were introduced in \cite{Kim1} as algebras $(X,\ra,1)$ of type $(2,0)$ satisfying the 
following conditions, for all $x,y,z\in X$: 
$(BE_1)$ $x\ra x=1;$ 
$(BE_2)$ $x\ra 1=1;$ 
$(BE_3)$ $1\ra x=x;$ 
$(BE_4)$ $x\ra (y\ra z)=y\ra (x\ra z)$. 
A relation $\le$ is defined on $X$ by $x\le y$ iff $x\ra y=1$. 
A BE algebra $X$ is \emph{bounded} if there exists $0\in X$ such that $0\le x$, for all $x\in X$. 
In a bounded BE algebra $(X,\ra,0,1)$ we define $x^*=x\ra 0$, for all $x\in X$. 
A bounded BE algebra $X$ is called \emph{involutive} if $x^{**}=x$, for any $x\in X$. \\
A BE algebra $X$ is called \emph{commutative} if $(x\ra y)\ra y=(y\ra x)\ra x$, for all $x,y\in X$. 
A bounded BE algebra $X$ is called \emph{involutive} if $x^{**}=x$, for any $x\in X$. \\
Obviously, any BCK algebra is a BE algebra, but the exact connection between BE algebras and 
BCK algebras is made in the papers \cite{Ior16, Ior17}: a BCK algebra is a BE algebra satisfying $(BCK_4)$ (antisymmetry) and $(BCK_1)$. \\
\indent
A \emph{suplement algebra} (\emph{S-algebra, for short}) is an algebra $(X,\oplus,^*,0,1)$ of type $(2,1,0,0)$   
satisfying the following axioms for all $x, y, z\in X$: 
$(S_1)$ $x\oplus y=y\oplus x;$ 
$(S_2)$ $x\oplus (y\oplus z)=(x\oplus y)\oplus z;$ 
$(S_3)$ $x\oplus x^*=1;$ 
$(S_4)$ $x\oplus 0=x;$ 
$(S_5)$ $x^{**}=x;$ 
$(S_6)$ $0^*=1;$ 
$(S_7)$ $x\oplus 1=1$ (\cite{Gudder}). \\
The following additional operations can be defined in a supplement algebra: \\
$x\odot y=(x^*\oplus y^*)^*$, $x\Cap_S y=(x\oplus y^*)\odot y$, $x\Cup_S y=(x\odot y^*)\oplus y$. \\
A \emph{quantum-MV algebra} (\emph{QMV algebra, for short}) is an S-algebra $(X,\oplus,^*,0,1)$ satisfying the 
following axiom for all $x, y, z\in X$ (\cite{Giunt2}): \\
$(QMV)$ $x\oplus ((x^*\Cap_S y)\Cap_S (z\Cap_S x^*))=(x\oplus y)\Cap_S (x\oplus z)$. 

\begin{lemma} \label{qbe-10} 
Let $(X,\ra,1)$ be a BE algebra. The following hold for all $x,y,z\in X$: \\
$(1)$ $x\ra (y\ra x)=1;$ \\
$(2)$ $x\le (x\ra y)\ra y$. \\
If $X$ is bounded, then: \\
$(3)$ $x\ra y^*=y\ra x^*;$ \\
$(4)$ $x\le x^{**}$. \\
If $X$ is involutive, then: \\
$(5)$ $x^*\ra y=y^*\ra x;$ \\
$(6)$ $x^*\ra y^*=y\ra x;$ \\
$(7)$ $(x\ra y)^*\ra z=x\ra (y^*\ra z);$ \\
$(8)$ $x\ra (y\ra z)=(x\ra y^*)^*\ra z;$ \\   
$(9)$ $(x^*\ra y)^*\ra (x^*\ra y)=(x^*\ra x)^*\ra (y^*\ra y)$.  
\end{lemma}
\begin{proof} 
$(1)$-$(6)$ See \cite{Ciu78}. \\ 
$(7)$ Applying $(BE_4)$ we get: $(x\ra y)^*\ra z=z^*\ra (x\ra y)=x\ra (z^*\ra y)=x\ra (y^*\ra z)$. \\
$(8)$ Using $(BE_4)$, we have: $x\ra (y\ra z)=x\ra (z^*\ra y^*)=z^*\ra (x\ra y^*)=(x\ra y^*)^*\ra z$. \\
$(9)$ Applying twice $(7)$, we get: 
$(x^*\ra y)^*\ra (x^*\ra y)=x^*\ra (y^*\ra (x^*\ra y))=x^*\ra (x^*\ra (y^*\ra y))=(x^*\ra x)^*\ra (y^*\ra y)$. 
\end{proof}

\noindent
In a BE algebra $X$, we define the additional operation $x\Cup y=(x\ra y)\ra y$. 
If $X$ is involutive, we define the operations $x\Cap y=((x^*\ra y^*)\ra y^*)^*=(x^*\Cup y^*)^*$, 
$x\odot y=(x\ra y^*)^*=(y\ra x^*)^*$, and the relation $\le_Q$ by $x\le_Q y$ iff $x=x\Cap y$. 

\begin{proposition} \label{qbe-20} Let $X$ be an involutive BE algebra. 
Then the following hold for all $x,y,z\in X$: \\
$(1)$ $x\le_Q y$ implies $x\le y$, $x=y\Cap x$ and $y=x\Cup y;$ \\
$(2)$ $\le_Q$ is reflexive and antisymmetric; \\
$(3)$ $(x\Cap y)\ra z=(y\ra x)\ra (y\ra z);$ \\
$(4)$ $(x\Cap y)^*\ra (y\ra x)^*=y\Cup (y\ra x)^*;$ \\
$(5)$ $(x\Cap (y\Cap z))^*=((z\ra x)\Cap (z\ra y))\ra z^*;$ \\
$(6)$ $x, y\le_Q z$ and $z\ra x=z\ra y$ imply $x=y;$ \emph{(cancellation law)} \\
$(7)$ $x\Cap y=y\odot (y\ra x)$.    
\end{proposition}
\begin{proof} 
$(1)-(3)$ See \cite{Ciu78}. \\
$(4)$ We have: \\
$\hspace*{2.00cm}$ $(x\Cap y)^*\ra (y\ra x)^*=((x^*\ra y^*)\ra y^*)\ra (y\ra x)^*$ \\
$\hspace*{5.50cm}$ $=((y\ra x)\ra y^*)\ra (y\ra x)^*$ \\
$\hspace*{5.50cm}$ $=(y\ra (y\ra x)^*)\ra (y\ra x)^*=y\Cup (y\ra x)^*$. \\
$(5)$ Aplying $(3)$, we get: \\
$\hspace*{1.00cm}$ $((z\ra x)\Cap (z\ra y))\ra z^*=((x^*\ra z^*)\Cap (y^*\ra z^*))\ra z^*$ \\
$\hspace*{5.25cm}$ $=((y^*\ra z^*)\ra (x^*\ra z^*))\ra ((y^*\ra z^*)\ra z^*)$ \\
$\hspace*{5.25cm}$ $=(x^*\ra ((y^*\ra z^*)\ra z^*))\ra ((y^*\ra z^*)\ra z^*)$ \\
$\hspace*{5.25cm}$ $=(x^*\ra (y^*\Cup z^*))\ra (y^*\Cup z^*)$ \\
$\hspace*{5.25cm}$ $=(x^*\ra (y\Cap x)^*)\ra (y\Cap z)^*$ \\
$\hspace*{5.25cm}$ $=(x\Cap (y\Cap z))^*$. \\
$(6)$ Since $x,y\le_Q z$ and $z\ra x=z\ra y$, we have: \\
$\hspace*{2.00cm}$ $x=x\Cap z=((x^*\ra z^*)\ra z^*)^*=((z\ra x)\ra z^*)^*$ \\
$\hspace*{2.30cm}$ $=((z\ra y)\ra z^*)^*=((y^*\ra z^*)\ra z^*)^*=y\Cap z=y$. \\
$(7)$ We have 
$y\odot (y\ra x)=(y\ra (y\ra x)^*)^*=((y\ra x)\ra y^*)^*=((x^*\ra y^*)\ra y^*)^*=x\Cap y$. 
\end{proof}

A \emph{(left-)quantum-Wajsberg algebra} (\emph{QW algebra, for short}) $(X,\ra,^*,1)$ is an involutive BE algebra $(X,\ra,^*,1)$ satisfying the following condition for all $x,y,z\in X$: \\
(QW) $x\ra ((x\Cap y)\Cap (z\Cap x))=(x\ra y)\Cap (x\ra z)$. 

\noindent
Condition (QW) is equivalent to the following conditions: \\
($QW_1$) $x\ra (x\Cap y)=x\ra y;$ \\ 
($QW_2$) $x\ra (y\Cap (z\Cap x))=(x\ra y)\Cap (x\ra z)$. 

\begin{definition} \label{qmv-30} $\rm($\cite{Ior30}$\rm)$
\emph{      
A \emph{(left-)m-BE algebra} is an algebra $(X,\odot,^{*},1)$ of type $(2,1,0)$ satisfying the following properties, 
for all $x,y,z\in X$:  
(PU) $1\odot x=x=x\odot 1;$ 
(Pcomm) $x\odot y=y\odot x;$ 
(Pass) $x\odot (y\odot z)=(x\odot y)\odot z;$  
(m-L) $x\odot 0=0;$ 
(m-Re) $x\odot x^{*}=0$, 
where $0:=1^*$. 
}\end{definition}
Note that, according to \cite[Cor. 17.1.3]{Ior35}, the involutive (left-)BE algebras $(X,\ra,^*,1)$ are term-equivalent to involutive (left-)m-BE algebras $(X,\odot,^*,1)$, by the mutually inverse transformations 
(\cite{Ior30, Ior35}): \\ 
$\hspace*{3cm}$ $\Phi:$\hspace*{0.2cm}$ x\odot y:=(x\ra y^*)^*$ $\hspace*{0.1cm}$ and  
                $\hspace*{0.1cm}$ $\Psi:$\hspace*{0.2cm}$ x\ra y:=(x\odot y^*)^*$. 
                
\begin{definition} \label{qmv-40} $\rm($\cite[Def. 3.10]{Ior34}$\rm)$
\emph{      
A \emph{(left-)quantum-MV algebra}, or a \emph{(left-)QMV algebra} for short, is an involutive (left-)m-BE algebra
$(X,\odot,^{*},1)$ verifying the following axiom: for all $x,y,z\in X$, \\
(Pqmv) $x\odot ((x^*\Cup y)\Cup (z\Cup x^*))=(x\odot y)\Cup (x\odot z)$. 
}
\end{definition}

\begin{proposition} \label{qmv-50}
The (left-)quantum-Wajsberg algebras are term-equivalent to (left-)quantum-MV algebras. 
\end{proposition}
\begin{proof}
We prove that the axioms $(Pqmv)$ and $(QW)$ are equivalent. 
Using the transformation $\Phi$, from (Pqmv) we get: \\
$\hspace*{0.5cm}$ $x\odot ((x^*\Cup y)\Cup (z\Cup x^*))=(x\ra ((x^*\Cup y)\Cup (z\Cup x^*))^*)^*
=(x\ra ((x\Cap y^*)\Cap (z^*\Cap x)))^*$ and \\
$\hspace*{0.5cm}$ $(x\odot y)\Cup (x\odot z)=(x\ra y^*)^*\Cup (x\ra z^*)^*=((x\ra y^*)\Cap (x\ra z^*))^*$, \\
hence (Pqmv) becomes: \\
$\hspace*{0.5cm}$ $(x\ra ((x\Cap y^*)\Cap (z^*\Cap x)))^*=((x\ra y^*)\Cap (x\ra z^*))^*$, \\
for all $x,y,z\in X$. Replacing $y$ by $y^*$ and $z$ by $z^*$, we get axiom (QW). 
Similarly axiom (QW) implies axiom (Pqmv). 
\end{proof}

In what follows, by quantum-MV algebras and quantum-Wajsberg algebras we understand the left-quantum-MV algebras and  left-quantum-Wajsberg algebras, respectively. 

\begin{proposition} \label{qbe-60} $\rm($\cite{Ciu78}$\rm)$ Let $X$ be a quantum-Wajsberg algebra. 
The following hold for all $x,y,z\in X$:\\
$(1)$ $x\ra (y\Cap x)=x\ra y$ and $(x\ra y)\ra (y\Cap x)=x;$ \\
$(2)$ $x\le_Q x^*\ra y$ and $x\le_Q y\ra x;$ \\
$(3)$ $x\le y$ iff $y\Cap x=x;$ \\
$(4)$ $(x\ra y)\Cup (y\ra x)=1$. \\
If  $x\le_Q y$, then: \\
$(5)$ $y=y\Cup x;$ \\
$(6)$ $y^*\le_Q x^*;$ \\
$(7)$ $y\ra z\le_Q x\ra z$ and $z\ra x\le_Q z\ra y;$ \\
$(8)$ $x\Cap z\le_Q y\Cap z$ and $x\Cup z\le_Q y\Cup z;$ \\
$(9)$ $x\odot z\le_Q y\odot z$. 
\end{proposition}

\begin{proposition} \label{qbe-80} Let $X$ be a quantum-Wajsberg algebra. The following hold, for all $x,y,z\in X$:\\
$(1)$ $(x\Cap y)\Cap (y\Cap z)=(x\Cap y)\Cap z;$ \\
$(2)$ $\le_Q$ is transitive; \\
$(3)$ $(z\Cap x)\ra (y\Cap x)=(z\Cap x)\ra y;$ \\ 
$(4)$ $x\le_Q y$ and $y\le x$ imply $x=y;$ \\
$(5)$ $x\le_Q y$ implies $x\Cap (y\Cap z)=x\Cap z;$ \\
$(6)$ $z\Cap ((y^*\ra z)\Cap (x^*\ra y))=z\Cap (x^*\ra y);$ \\ 
$(7)$ $x\Cup (x\ra y)^*=x;$ \\
$(8)$ $x=y\ra x$ iff $y=x\ra y;$ \\
$(9)$ $x\Cap y, y\Cap x\le_Q x\ra y$. 
\end{proposition}
\begin{proof}
$(1)-(3)$ See \cite{Ciu78}. \\
$(4)$ By Proposition \ref{qbe-60}$(3)$, $y\le x$ implies $x\Cap y=y$. 
Since $x\le_Q y$, we have $x\Cap y=x$, hence $x=y$. \\
$(5)$ Using $(1)$, $(x\Cap y)\Cap (y\Cap z)=(x\Cap y)\Cap z$.  
Since $x\le_Q y$ implies $x\Cap y=x$, we get $x\Cap (y\Cap z)=x\Cap z$. \\
$(6)$ It follows by $(5)$, since $z\le_Q y^*\ra z;$ \\
$(7)$ By Proposition \ref{qbe-60}$(2)$,$(5)$, we have $x^*\le_Q x\ra y$, so that $(x\ra y)^*\le_Q x$ and 
$x\Cup (x\ra y)^*=x$. \\
$(8)$ Suppose $x=y\ra x$, so that $y^*\Cap x^*=(y\Cup x)^*=((y\ra x)\ra x)^*=(x\ra x)^*=1^*=0$. 
Using $(QW_1)$, we get $y=(y^*)^*=y^*\ra 0=y^*\ra (y^*\Cap x^*)=y^*\ra x^*=x\ra y$. 
The converse follows similarly. \\
$(9)$ Using Proposition \ref{qbe-60}$(2)$, we have $y\le_Q x\ra y$, so that 
$(x\ra y)^*\le_Q y^*\le_Q (x^*\ra y^*)\ra y^*=(x\Cap y)^*$. Hence $x\Cap y\le_Q x\ra y$. 
Similarly $(x\ra y)^*\le (x\ra y)\ra x^*=(y^*\ra x^*)\ra x^*=(y\Cap x)^*$. 
Thus $y\Cap x\le_Q x\ra y$. 
\end{proof}
\noindent
By Propositions \ref{qbe-20}$(2)$, \ref{qbe-80}$(2)$, in a quantum-Wajsberg algebra $X$, $\le_Q$ is a partial 
order on $X$. \\
A quantum-Wajsberg algebra $X$ is called \emph{commutative} if $x\Cup y=y\Cup x$, or equivalently 
$x\Cap y=y\Cap x$ for all $x,y\in X$. \\
Since: \\
- commutative BE algebras are commutative BCK algebras (\cite{Walend1}]),  \\
- bounded commutative BCK are term-equivalent to MV algebras (\cite{Mund1}) and \\
- Wajsberg algebras are term-equivalent to MV algebras (\cite{Font1}), \\
it follows that bounded commutative BE algebras are bounded commutative BCK algebras, hence are term-equivalent 
to MV algebras, hence to Wajsberg algebras. \\
Hence the commutative quantum-Wajsberg algebras are the Wajsberg algebras. \\
It was proved in \cite{Ciu78} that a quantum-Wajsberg algebra is a bounded commutative BCK algebra, that is a 
Wajsberg algebra, if and only if the relations $\le$ and $\le_Q$ coincide. 

\begin{proposition} \label{qbe-90} $\rm($\cite{Ciu78}$\rm)$ Let $(X,\ra,0,1)$ be a bounded commutative BCK algebra. 
The following hold for all $x,y,z\in X$: \\
$(1)$ $x\le_Q y$ and $x\le_Q z$ imply $x\le_Q y\Cap z;$ \\
$(2)$ $y\le_Q x$ and $z\le_Q x$ imply $y\Cup z\le_Q x;$ \\
$(3)$ $x\le_Q y$ implies $x\Cup z\le_Q y\Cup z$ and $x\Cap z\le_Q y\Cap z$. 
\end{proposition}

$\vspace*{2mm}$

\section{The Wajsberg-center of quantum-Wajsberg algebras}

In this section,  we investigate the commutativity property of quantum-Wajsberg algebras. 
We define the Wajsberg-center or the commutative center of a quantum-Wajsberg algebra $X$ as the set of those elements 
of $X$ that commute with all other elements of $X$. 
We study certain properties of the Wajsberg-center, and prove that the Wajsberg-center is a Wajsberg  
subalgebra of $X$. 
In what follows, $(X,\ra,^*,1)$ will be a quantum-Wajsberg algebra, unless otherwise stated. 

\begin{definition} \label{cqbe-10}
\emph{
We say that the elements $x,y\in X$ \emph{commute}, denoted by $x \mathcal{C} y$, if $x\Cap y=y\Cap x$. 
}
\end{definition}

\begin{definition} \label{cqbe-10-10}
\emph{
The \emph{commutative center} of $X$ is the set $\mathcal{Z}(X)=\{x\in X\mid x \mathcal{C} y$, for all $y\in X\}$.  
}
\end{definition}

\noindent
Obviously $0,1\in \mathcal{Z}(X)$. 

\begin{lemma} \label{cqbe-20} If $x \mathcal{C} y$, then $x\Cup y=y\Cup x$. 
\end{lemma}
\begin{proof}
Applying twice Proposition \ref{qbe-60}$(1)$, we have: \\
$\hspace*{2.00cm}$ $x\Cup y=(x\ra y)\ra y= (x\ra y) \ra ((y\ra x)\ra (x\Cap y))$ \\
$\hspace*{2.95cm}$ $= (y\ra x)\ra ((x\ra y)\ra (x\Cap y))$ \\
$\hspace*{2.95cm}$ $= (y\ra x)\ra ((x\ra y)\ra (y\Cap x))=(y\ra x)\ra x=y\Cup x$. 
\end{proof}

\begin{lemma} \label{cqbe-30} Let $x, y\in X$. The following are equivalent: \\
$(a)$ $x \mathcal{C} y;$ \\
$(b)$ $(x\ra y)\ra (x\Cap y)=x$. 
\end{lemma}
\begin{proof}
$(a)\Rightarrow (b)$ By Proposition \ref{qbe-60}$(1)$, we get  
$x=(x\ra y)\ra (y\Cap x)=(x\ra y)\ra (x\Cap y)$. \\
$(b)\Rightarrow (a)$ Suppose $(x\ra y)\ra (x\Cap y)=x$, and applying Proposition \ref{qbe-60}$(1)$, we have: 
$(x\ra y)\ra (x\Cap y)=(x\ra y)\ra (y\Cap x)(=x)$. 
Since by Proposition \ref{qbe-80}$(9)$, $x\Cap y, y\Cap x\le_Q x\ra y$, by cancellation law (Proposition \ref{qbe-20}$(6)$), we get $x\Cap y=y\Cap x$. Hence $x \mathcal{C} y$. 
\end{proof}

\begin{proposition} \label{cqbe-40} The following hold: \\
$(1)$ the relation $\mathcal{C}$ is reflexive and symmetric; \\
$(2)$ if $x\le_Q y$ or $y\le_Q x$, then $x \mathcal{C} y;$ \\ 
$(3)$ $x \mathcal{C} y$ implies $x^* \mathcal{C} y^*;$ \\
$(4)$ $(x\Cap y)^* \mathcal{C} (x\ra y)^*$. 
\end{proposition}
\begin{proof}
$(2)$ If $x\le_Q y$, then $x=x\Cap y$ and, by Proposition \ref{qbe-20}$(1)$ we have $x=y\Cap x$. 
Hence $x\Cap y=y\Cap x$, that is $x \mathcal{C} y$, and similarly $y\le_Q x$ implies $x \mathcal{C} y$. \\  
$(3)$ Using Lemma \ref{cqbe-20}, we have: $x^*\Cap y^*=(x\Cup y)^*=(y\Cup x)^*=y^*\Cap x^*$, 
hence $x^* \mathcal{C} y^*$. \\
$(4)$ Since $x\Cap y\le_Q y\le_Q x\ra y$, we get $(x\ra y)^*\le_Q (x\Cap y)^*$. 
Applying $(2)$, it follows that $(x\Cap y)^* \mathcal{C} (x\ra y)^*$. 
\end{proof}

\begin{corollary} \label{cqbe-40-10} $\mathcal{Z}(X)$ is closed under $^*$. 
\end{corollary}
\begin{proof}
Let $x\in \mathcal{Z}(X)$, that is $x\mathcal{C} z$ for all $z\in X$. 
We also have $x\mathcal{C} z^*$, and applying Lemma \ref{cqbe-20} we have $x\Cup z^*=z^*\Cup x$. 
It follows that $x^*\Cap z=(x\Cup z^*)^*=(z^*\Cup x)^*=z\Cap x^*$. Hence $x^*\in \mathcal{Z}(X)$.  
\end{proof}

\begin{proposition} \label{cqbe-50} If $x,y,z\in X$ such that $x \mathcal{C} y$ and $x \mathcal{C} z$, then 
$(x\Cap y)\Cap z=y\Cap (x\Cap z)$. 
\end{proposition}
\begin{proof} 
From $x\Cap y=y\Cap x$, $x\Cap z=z\Cap x$, and applying Proposition \ref{qbe-80}$(1)$,$(3)$, we get: \\
$\hspace*{2.00cm}$ $(x\Cap y)\Cap z=(y\Cap x)\Cap z=(y\Cap x)\Cap (x\Cap z)$ \\
$\hspace*{3.90cm}$ $=(((y\Cap x)^*\ra (x\Cap z)^*)\ra (x\Cap z)^*)^*$ \\
$\hspace*{3.90cm}$ $=(((y\Cap x)^*\ra (z\Cap x)^*)\ra (z\Cap x)^*)^*$ \\
$\hspace*{3.90cm}$ $=(((z\Cap x)\ra (y\Cap x))\ra (z\Cap x)^*)^*$ \\
$\hspace*{3.90cm}$ $=(((z\Cap x)\ra y)\ra (z\Cap x)^*)^*$ \\
$\hspace*{3.90cm}$ $=((y^*\ra (z\Cap x)^*)\ra (z\Cap x)^*)^*$ \\
$\hspace*{3.90cm}$ $=y\Cap (z\Cap x)=y\Cap (x\Cap z)$. 
\end{proof}

\begin{corollary} \label{cqbe-60} If $x \mathcal{C} y$, $y \mathcal{C} z$ and $x \mathcal{C} z$, then 
$(x\Cap y)\Cap z=z\Cap (x\Cap y)$. 
\end{corollary}
\begin{proof} 
By hypothesis and using Proposition \ref{cqbe-50}, we get: \\
$\hspace*{2.00cm}$ $(x\Cap y)\Cap z=y\Cap (x\Cap z)=y\Cap (z\Cap x)=z\Cap (y\Cap x)=z\Cap (x\Cap y)$. 
\end{proof}

\begin{corollary} \label{cqbe-70} $\mathcal{Z}(X)$ is closed under $\Cap$. 
\end{corollary}
\begin{proof} 
Let $x,y \in \mathcal{Z}(X)$ and let $z \in X$. It follows that $x \mathcal{C} y$, $y \mathcal{C} z$, 
$x \mathcal{C} z$, and by Corollary \ref{cqbe-60} we get $(x\Cap y)\Cap z=z\Cap (x\Cap y)$. 
Hence $x\Cap y \in \mathcal{Z}(X)$, that is $\mathcal{Z}(X)$ is closed under $\Cap$. 
\end{proof}

\begin{proposition} \label{cqbe-80} Let $x,y,z\in X$ such that $y \mathcal{C} z$. 
Then $x\ra (y\Cap z)\le_Q (x\ra y)\Cap (x\ra z)$.  
\end{proposition}
\begin{proof} 
From $z\Cap y\le_Q y$, we get $x\ra (z\Cap y)\le_Q x\ra y$, so that $(x\ra y)^*\le_Q (x\ra (z\Cap y))^*$ and 
$(x\ra (z\Cap y))^*\ra (x\ra z)^*\le_Q (x\ra y)^*\ra (x\ra z)^*$. 
It follows that: \\
$\hspace*{0.50cm}$
$((x\ra (z\Cap y))^*\ra (x\ra z)^*)\Cap ((x\ra y)^*\ra (x\ra z)^*)=(x\ra (z\Cap y))^*\ra (x\ra z)^*$. \\
Similarly, from $y\Cap z\le_Q z$ we have $x\ra (y\Cap z)\le_Q x\ra z$, hence 
$(x\ra (y\Cap z))\Cap (x\ra z)=x\ra (y\Cap z)$. 
Applying Proposition \ref{qbe-20}$(5)$, and taking into consideration that $y \mathcal{C} z$, we have: \\
$\hspace*{0.50cm}$ $(x\ra (y\Cap z))\Cap ((x\ra y)\Cap (x\ra z))=$ \\
$\hspace*{1.00cm}$ 
$=((((x\ra z)\ra (x\ra (y\Cap z)))\Cap ((x\ra z)\ra (x\ra y)))\ra (x\ra z)^*)^*$ \\ 
$\hspace*{1.00cm}$
$=((((x\ra (y\Cap z))^*\ra (x\ra z)^*)\Cap ((x\ra y)^*\ra (x\ra z)^*))\ra (x\ra z)^*)^*$ \\
$\hspace*{1.00cm}$ 
$=((((x\ra (z\Cap y))^*\ra (x\ra z)^*)\Cap ((x\ra y)^*\ra (x\ra z)^*))\ra (x\ra z)^*)^*$ \\
$\hspace*{1.00cm}$
$=(((x\ra (z\Cap y))^*\ra (x\ra z)^*)\ra (x\ra z)^*)^*$ \\
$\hspace*{1.00cm}$
$=(((x\ra (y\Cap z))^*\ra (x\ra z)^*)\ra (x\ra z)^*)^*$ \\
$\hspace*{1.00cm}$
$=(x\ra (y\Cap z))\Cap (x\ra z)=x\ra (y\Cap z)$. \\
Hence $x\ra (y\Cap z)\le_Q (x\ra y)\Cap (x\ra z)$. 
\end{proof}

\begin{lemma} \label{cqbe-90} If $x \mathcal{C} y$, $x \mathcal{C} z$, $y \mathcal{C} z$, 
then $y\Cap (z\Cap x)\le_Q y\Cap z$. 
\end{lemma}
\begin{proof} 
From $z\Cap x=x\Cap z\le_Q z$, by Proposition \ref{qbe-60}$(8)$ we get $(z\Cap x)\Cap y\le_Q z\Cap y=y\Cap z$. 
Using Corollary \ref{cqbe-60}, we get $y\Cap (z\Cap x)\le_Q y\Cap z$. 
\end{proof}

\begin{proposition} \label{cqbe-100} If $x \mathcal{C} y$, $x \mathcal{C} z$, $y \mathcal{C} z$, 
then $(x\ra y)\Cap (x\ra z) \le x\ra (y\Cap z)$.  
\end{proposition}
\begin{proof} 
Applying Proposition \ref{qbe-20}$(3)$ and Lemma \ref{cqbe-90}, we get: \\
$\hspace*{0.50cm}$ $((x\ra y)\Cap (x\ra z))\ra (x\ra (y\Cap z))=$ \\
$\hspace*{4.00cm}$ $=(x\ra (y\Cap z))^*\ra ((x\ra y)\Cap (x\ra z))^*$ \\
$\hspace*{4.00cm}$ $=(x\ra (y\Cap z))^*\ra ((x\ra y)^*\Cup (x\ra z)^*)$ \\
$\hspace*{4.00cm}$ $=(x\ra (y\Cap z))^*\ra (((x\ra y)^*\ra (x\ra z)^*)\ra (x\ra z)^*)$ \\
$\hspace*{4.00cm}$ $=((x\ra y)^*\ra (x\ra z)^*)\ra ((x\ra (y\Cap z))^*\ra (x\ra z)^*)$ \\
$\hspace*{4.00cm}$ $=((x\ra z)\ra (x\ra y))\ra ((x\ra z)\ra (x\ra (y\Cap z)))$ \\
$\hspace*{4.00cm}$ $=((z\Cap x)\ra y)\ra ((z\Cap x)\ra (y\Cap z))$ \\ 
$\hspace*{4.00cm}$ $=(y\Cap (z\Cap x))\ra (y\Cap z)=1$. \\
It follows that $(x\ra y)\Cap (x\ra z)\le x\ra (y\Cap z)$. 
\end{proof}

\begin{proposition} \label{cqbe-110} If $x \mathcal{C} y$, $x \mathcal{C} z$, $y \mathcal{C} z$, then  
$x\ra (y\Cap z)=(x\ra y)\Cap (x\ra z)$.  
\end{proposition}
\begin{proof}
It follows by Propositions \ref{cqbe-80}, \ref{cqbe-100}, \ref{qbe-80}$(4)$. 
\end{proof}

\begin{corollary} \label{cqbe-120} If $y\in \mathcal{Z}(X)$ and $x,z\in X$, then  
$(z\Cap x)^*\ra ((z\ra x)^*\Cap y)=((z\Cap x)^*\ra (z\ra x)^*)\Cap ((z\Cap x)^*\ra y)$.    
\end{corollary}
\begin{proof}
It follows by Propositions \ref{cqbe-40}$(4)$ and \ref{cqbe-110}, since $y\mathcal{C} (z\Cap x)^*$ and 
$y\mathcal{C} (z\ra x)^*$. 
\end{proof}

\begin{corollary} \label{cqbe-130} If $x,y,z\in \mathcal{Z}(X)$, then  
$(z\Cap x)^*\ra y=(y^*\ra z)\Cap (x^*\ra y)$. 
\end{corollary}
\begin{proof}
Since $y\in \mathcal{Z}(X)$ implies $y^*\in \mathcal{Z}(X)$, applying Proposition \ref{cqbe-110}, we get: 
$(z\Cap x)^*\ra y=y^*\ra (z\Cap x)=(y^*\ra z)\Cap (y^*\ra x)=(y^*\ra z)\Cap (x^*\ra y)$. 
\end{proof}

\begin{proposition} \label{cqbe-140} If $x,y\in \mathcal{Z}(X)$, then $x^*\ra y\in \mathcal{Z}(X)$. 
\end{proposition}
\begin{proof} 
Let $x,y\in \mathcal{Z}(X)$ and let $z\in X$. Then $x \mathcal{C} z$ and $y \mathcal{C} (z\ra x)^*$. 
Applying Lemma \ref{cqbe-30}, we get: \\ 
$\hspace*{2.00cm}$ $z=(z\ra x)\ra (z\Cap x)=((z\ra x)^*)^*\ra (z\Cap x)$ and \\
$\hspace*{2.00cm}$ $(z\ra x)^*=((z\ra x)^*\ra y)\ra ((z\ra x)^*\Cap y)$, \\
respectively. It follows that: \\ 
$\hspace*{1.00cm}$ $z=((z\ra x)^*)^*\ra (z\Cap x)$ \\
$\hspace*{1.30cm}$ $=(((z\ra x)^*\ra y)\ra ((z\ra x)^*\Cap y))^*\ra (z\Cap x)$ \\
$\hspace*{1.30cm}$ $=((z\ra x)^*\ra y)\ra (((z\ra x)^*\Cap y)^*\ra (z\Cap x))$ (by Lemma \ref{qbe-10}$(7)$) \\
$\hspace*{1.30cm}$ $=((z\ra x)^*\ra y)\ra ((z\Cap x)^*\ra ((z\ra x)^*\Cap y))$ \\
$\hspace*{1.30cm}$ $=((z\ra x)^*\ra y)\ra (((z\Cap x)^*\ra (z\ra x)^*)\Cap ((z\Cap x)^*\ra y))$ 
                                                                   (by Corollary \ref{cqbe-120}) \\
$\hspace*{1.30cm}$ $=((z\ra x)^*\ra y)\ra (((x\Cap z)^*\ra (z\ra x)^*)\Cap ((z\Cap x)^*\ra y))$ \\                      $\hspace*{1.30cm}$ $=((z\ra x)^*\ra y)\ra ((z\Cup (z\ra x)^*)\Cap ((z\Cap x)^*\ra y))$ 
                                                                   (by Proposition \ref{qbe-20}$(4)$) \\
$\hspace*{1.30cm}$ $=((z\ra x)^*\ra y)\ra (z\Cap ((z\Cap x)^*\ra y))$
                                                                   (by Proposition \ref{qbe-80}$(7)$) \\
$\hspace*{1.30cm}$ $=((z\ra x)^*\ra y)\ra (z\Cap ((y^*\ra z)\Cap (x^*\ra y))$
                                                                   (by Corrolary \ref{cqbe-130}) \\
$\hspace*{1.30cm}$ $=((z\ra x)^*\ra y)\ra (z\Cap (x^*\ra y))$      (by Proposition \ref{qbe-80}$(6)$) \\
$\hspace*{1.30cm}$ $=(z\ra (x^*\ra y))\ra (z\Cap (x^*\ra y))$      (by Lemma \ref{qbe-10}$(7)$). \\
Using Lemma \ref{cqbe-30}, we conclude that $x^*\ra y\in \mathcal{Z}(X)$. 
\end{proof}

\begin{corollary} \label{cqbe-150} If $x,y\in \mathcal{Z}(X)$, then $x\ra y\in \mathcal{Z}(X)$. 
\end{corollary}
\begin{proof} 
Since $x\in \mathcal{Z}(X)$, by Corollary \ref{cqbe-40-10} we get $x^*\in \mathcal{Z}(X)$. 
Applying Proposition \ref{cqbe-140}, $x^*,y\in \mathcal{Z}(X)$ implies $x\ra y\in \mathcal{Z}(X)$. 
\end{proof}

\begin{theorem} \label{cqbe-160} $(\mathcal{Z}(X),\ra,0,1)$ is a Wajsberg subalgebra of $X$.    
\end{theorem}
\begin{proof} 
Since by Corollary \ref{cqbe-150}, $x,y\in \mathcal{Z}(X)$ implies $x\ra y\in \mathcal{Z}(X)$, 
it follows that $\mathcal{Z}(X)$ is closed under $\ra$. 
Moreover $0,1\in \mathcal{Z}(X)$, hence it is a quantum-Wajsberg subalgebra of $X$. 
Since $x,y\in \mathcal{Z}(X)$ implies $x\Cap y=y\Cap x$, $(\mathcal{Z}(X),\ra,0,1)$ is a commutative quantum-Wajsberg algebra, that is a bounded commutative BCK subalgebra of $X$. Hence it is a Wajsberg subalgebra of $X$. 
\end{proof}

\begin{corollary} \label{cqbe-170} A quantum-Wajsberg algebra $X$ is a Wajsberg algebra if and only if  $\mathcal{Z}(X)=X$.    
\end{corollary}


Taking into consideration the above results, the commutative center $\mathcal{Z}(X)$ will be also called the 
\emph{Wajsberg-center} of $X$. 
Similarly as in \cite{Giunt1} for the case of QMV algebras, we define the notion of a quasi-linear quantum-Wajsberg algebra. 

\begin{definition} \label{cqbe-180-10}
\emph{
A QW algebra $X$ is said to be \emph{quasi-linear} if, for all $x,y\in X$, $x\nleq_Q y$ implies $y< x$. 
}
\end{definition}

\begin{proposition} \label{cqbe-190} If $X$ is a quasi-linear QW algebra, then $\mathcal{Z}(X)$ is a linearly 
ordered Wajsberg algebra. 
\end{proposition}
\begin{proof}
According to \cite{Ciu78}, a quantum-Wajsberg algebra is a Wajsberg algebra if and only if the 
relations $\le$ and $\le_Q$ coincide. Since $\mathcal{Z}(X)$ is a quasi-linear Wajsberg algebra, 
$x\nleq y$ implies $y< x$, that is $\mathcal{Z}(X)$ is linearly ordered. 
\end{proof}

$\vspace*{1mm}$

\section{The lattice structure of Wajsberg-centers}

We study certain lattice properties of the Wajsberg-center of a quantum-Wajsberg algebra, and prove that the 
Wajsberg-center of a quantum-Wajsberg algebra $X$ is a distributive sublattice of the poset $(X,\le_Q,0,1)$.   
If the quantum-Wajsberg algebra is quasi-linear, we prove that the Wajsberg-center is a linearly ordered 
Wajsberg algebra. Finally, we show that the lattice subreduct of the Wajsberg-center is a Kleene algebra. 
In what follows, $(X,\ra,^*,1)$ will be a quantum-Wajsberg algebra, unless otherwise stated. 

\begin{proposition} \label{lqbe-10} The following hold for all $x,y,z\in \mathcal{Z}(X)$: \\
$(1)$ $x\ra (y\Cap z)=(x\ra y)\Cap (x\ra z)$ $($distributivity of $\ra$ over $\Cap);$ \\
$(2)$ $x\odot (y\Cup z)=(x\odot y)\Cup (x\odot z)$ $($distributivity of $\odot$ over $\Cup);$ \\
$(3)$ $x\Cap (y\Cup z)=(x\Cap y)\Cup (x\Cap z)$ $($distributivity of $\Cap$ over $\Cup);$ \\
$(4)$ $x\Cup (y\Cap z)=(x\Cup y)\Cap (x\Cup z)$ $($distributivity of $\Cup$ over $\Cap)$. 
\end{proposition}
\begin{proof}
$(1)$ It follows by Proposition \ref{cqbe-110}. \\
$(2)$ Applying $(1)$, we get: \\ 
$\hspace*{2.00cm}$ $x\odot (y\Cup z)=(x\ra (y\Cup z)^*)^*=(x\ra (y^*\Cap z^*))^*$ \\
$\hspace*{4.00cm}$ $=((x\ra y^*)\Cap (x\ra z^*))^*=(x\ra y^*)^*\Cup (x\ra z^*)^*$ \\
$\hspace*{4.00cm}$ $=(x\odot y)\Cup (x\odot z)$. \\
$(3)$ By commutativity we have $y,z\le_Q y\Cup z$, so that $(y\Cup z)\ra x\le_Q y\ra x, z\ra x$.  
Applying Propositions \ref{qbe-60}$(9)$ and \ref{qbe-20}$(7)$, we get: \\ 
$\hspace*{2.00cm}$ $y\odot ((y\Cup z)\ra x)\le_Q y\odot (y\ra x)=x\Cap y$ and \\
$\hspace*{2.00cm}$ $z\odot ((y\Cup z)\ra x)\le_Q z\odot (z\ra x)=x\Cap z$. \\
Using Proposition \ref{qbe-20}$(7)$ and $(2)$, we have: \\
$\hspace*{1.00cm}$ $x\Cap (y\Cup z)=(y\Cup z)\odot ((y\Cup z)\ra x)=
    (y\odot ((y\Cup z)\ra x))\Cup (z\odot ((y\Cup z)\ra x))$ \\
$\hspace*{2.90cm}$ $\le_Q (x\Cap y)\Cup (x\Cap z)$. \\
On the other hand, $x\Cap y\le_Q y$, $x\Cap z\le_Q z$ imply $(x\Cap y)\Cup (x\Cap z)\le_Q y\Cup z$, and 
$x\Cap y\le_Q x$, $x\Cap z\le_Q x$ imply $(x\Cap y)\Cup (x\Cap z)\le_Q x$. 
Hence by Proposition \ref{qbe-90}, $(x\Cap y)\Cup (x\Cap z)\le_Q x\Cap (y\Cup z)$. 
Since $\mathcal{Z}(X)$ is a commutative bounded BCK algebra, the relation $\le_Q$ is antisymmetric, and 
we conclude that $x\Cap (y\Cup z)=(x\Cap y)\Cup (x\Cap z)$. \\
$(4)$ Applying $(3)$, we have: \\
$\hspace*{2.10cm}$ $x\Cup (y\Cap z)=(x^*\Cap (y\Cap z)^*)^*=(x^*\Cap (y^*\Cup z^*))^*$ \\
$\hspace*{4.00cm}$ $=((x^*\Cap y^*)\Cup (x^*\Cap z^*))^*=((x\Cup y)^*\Cup (x\Cup z)^*)^*$ \\
$\hspace*{4.00cm}$ $=(x\Cup y)\Cap (x\Cup z)$. 
\end{proof}

\begin{lemma} \label{lqbe-20} The following hold for all $x,y\in \mathcal{Z}(X)$: \\
$(1)$ $x\Cup y$ is the least upper bound (l.u.b.) of $\{x,y\};$ \\
$(2)$ $x\Cap y$ is the greatest lower bound (g.l.b.) of $\{x,y\}$. 
\end{lemma}
\begin{proof}
$(1)$ By Corollaries \ref{cqbe-70} and \ref{cqbe-40-10}, $\mathcal{Z}(X)$ is closed under $\Cap$ and ${}^*$. 
Since $x\Cup y=(x^*\Cap y^*)^*$ for all $x,y\in \mathcal{Z}(X)$, it follows that $\mathcal{Z}(X)$ is also closed 
under $\Cup$.  
Since by commutativity $x,y\le_Q x\Cup y$, it follows that $x\Cup y$ is an upper bound of $\{x,y\}$. 
Let $z$ be another upper bound of $\{x,y\}$, so that $x,y\le_Q z$, that is $x=x\Cap z$ and $y=y\Cap z$. 
Using Proposition \ref{lqbe-10}$(3)$, we have 
$(x\Cup y)\Cap z=z\Cap (x\Cup y)=(z\Cap x)\Cup (z\Cap y)=(x\Cap z)\Cup (y\Cap z)=x\Cup y$. 
Hence $x\Cup y\le_Q z$, so that $x\Cup y$ is the l.u.b. of $\{x,y\}$. \\
$(2)$ By commutativity we also have $x\Cap y\le_Q x,y$, thus $x\Cap y$ is a lower bound of $\{x,y\}$. 
Let $z$ be another lower bound of $\{x,y\}$, so that $z\le_Q x,y$, that is $z=z\Cap x$ and $z=z\Cap y$. 
Using Proposition \ref{cqbe-50} and Corollary \ref{cqbe-60}, we have: 
$z\Cap (x\Cap y)=(x\Cap y)\Cap z=y\Cap (x\Cap z)=y\Cap (z\Cap x)=y\Cap z=z\Cap y=z$, that is $z\le_Q x\Cap y$. 
It follows that $x\Cap y$ is the g.l.b. of $\{x,y\}$. 
\end{proof}

\begin{theorem} \label{lqbe-20-10} $(\mathcal{Z}(X),\Cap,\Cup,0,1)$ is a distributive sublattice of the poset 
$(X,\le_Q,0,1)$. 
\end{theorem}
\begin{proof} It follows by Lemma \ref{lqbe-20}, Theorem \ref{cqbe-160} and Proposition \ref{lqbe-10}. 
\end{proof}

\begin{proposition} \label{lqbe-30} The following hold for all $x,y,z\in \mathcal{Z}(X)$: \\
$(1)$ $x\ra (y\Cup z)=(x\ra y)\Cup (x\ra z)$ $($distributivity of $\ra$ over $\Cup);$ \\
$(2)$ $x\odot (y\Cap z)=(x\odot y)\Cap (x\Cup z)$ $($distributivity of $\odot$ over $\Cap);$ \\
$(3)$ $(y\Cup z)\ra x=(y\ra x)\Cap (z\ra x);$ \\
$(4)$ $(y\Cap z)\ra x=(y\ra x)\Cup (z\ra x)$.   
\end{proposition}
\begin{proof}
$(1)$ Since by commutativity $y\Cup z\ge_Q y,z$, we have $x\ra (y\Cup z)\ge_Q x\ra y, x\ra z$, so that 
$x\ra (y\Cup z)$ is an upper bound of $\{x\ra y, x\ra z\}$. 
Let $u$ be another upper bound of $\{x\ra y, x\ra z\}$, that is $u\ge_Q x\ra y, x\ra z$. 
It follows that $u\odot x\ge_Q (x\ra y)\odot x=x\Cap y$ and $u\odot x\ge_Q (x\ra z)\odot x=x\Cap z$.  
Hence, by Proposition \ref{lqbe-10}$(3)$, $x\Cap (y\Cup z)=(x\Cap y)\Cup (x\Cap z)\le_Q u\odot x$. 
Using $(QW_1)$, we get 
$x\ra (y\Cup z)=x\ra (x\Cap (y\Cup z))\le_Q x\ra (u\odot x)=x\ra (u\ra x^*)^*=(u\ra x^*)\ra x^*=u\Cup x^*=u$ 
(since $u\ge_Q x\ra y\ge_Q x^*)$. 
Thus $x\ra (y\Cup z)$ is the least upper bound of $\{x\ra y, x\ra z\}$, and so 
$x\ra (y\Cup z)=(x\ra y)\Cup (x\ra z)$. \\
$(2)$ Using $(1)$, we have: \\
$\hspace*{2.00cm}$ $x\odot (y\Cap z)=(x\ra (y\Cap z)^*)^*=(x\ra (y^*\Cup z^*))^*$ \\
$\hspace*{4.00cm}$ $=((x\ra y^*)\Cup (x\ra z^*))^*=((x\odot y)^*\Cup (x\odot z)^*)^*$ \\
$\hspace*{4.00cm}$ $=(x\odot y)\Cap (x\odot z)$. \\
$(3)$ Applying Proposition \ref{lqbe-10}$(1)$, we have: \\
$\hspace*{2.00cm}$ $(y\Cup z)\ra x=x^*\ra (y\Cup z)^*=x^*\ra (y^*\Cap z^*)$ \\
$\hspace*{4.00cm}$ $=(x^*\ra y^*)\Cap (x^*\ra z^*)=(y\ra x)\Cap (z\ra x)$. \\
$(4)$ By $(1)$, we get: \\
$\hspace*{2.00cm}$ $(y\Cap z)\ra x=x^*\ra (y\Cap z)^*=x^*\ra (y^*\Cup z^*)$ \\
$\hspace*{4.00cm}$ $=(x^*\ra y^*)\Cup (x^*\ra z^*)=(y\ra x)\Cup (z\ra x)$.  
\end{proof}

\begin{proposition} \label{lqbe-40} The following hold for all $x,y,z\in \mathcal{Z}(X)$: \\
$(1)$ $(x\Cup y)\ra (x\Cup z)\ge_Q x\Cup (y\ra z);$ \\
$(2)$ $(x\Cap y)\ra (x\Cap z)\ge_Q x\Cap (y\ra z)$. 
\end{proposition}
\begin{proof} 
$(1)$ Applying Proposition \ref{lqbe-30}, since $y\ra x\ge_Q x$ we get: \\
$\hspace*{1.00cm}$ $(x\Cup y)\ra (x\Cup z)=(x\ra (x\Cup z))\Cap (y\ra (x\Cup z))$ \\
$\hspace*{4.00cm}$ $=((x\ra x)\Cup (x\ra z))\Cap ((y\ra x)\Cup (y\ra z))$ \\
$\hspace*{4.00cm}$ $=(1\Cup (x\ra z))\Cap ((y\ra x)\Cup (y\ra z))$ \\
$\hspace*{4.00cm}$ $=1\Cap ((y\ra x)\Cup (y\ra z))$ \\
$\hspace*{4.00cm}$ $=(y\ra x)\Cup (y\ra z)\ge_Q x\Cup (y\ra z)$. \\
$(2)$ Similarly, using Proposition \ref{lqbe-10} we have: \\
$\hspace*{1.00cm}$ $(x\Cap y)\ra (x\Cap z)=(x\ra (x\Cap z))\Cup (y\ra (x\Cap z))$ \\
$\hspace*{4.00cm}$ $=((x\ra x)\Cap (x\ra z))\Cup ((y\ra x)\Cap (y\ra z))$ \\
$\hspace*{4.00cm}$ $=(1\Cap (x\ra z))\Cup ((y\ra x)\Cap (y\ra z))$ \\
$\hspace*{4.00cm}$ $=(x\ra z)\Cup ((y\ra x)\Cap (y\ra z))$ \\
$\hspace*{4.00cm}$ $\ge_Q (y\ra x)\Cap (y\ra z)\ge_Q x\Cap (y\ra z)$. 
\end{proof}

\begin{proposition} \label{lqbe-50} The following hold for all $x,y\in \mathcal{Z}(X)$: \\
$(1)$ $(x^*\odot y)\Cap (x\odot y^*)=0;$ \\
$(2)$ $(x\Cap x^*)\odot (y\Cap y^*)=0;$ \\
$(3)$ $x\Cap x^*\le_Q y\Cup y^*$.  
\end{proposition}
\begin{proof} 
$(1)$ Applying Proposition \ref{qbe-60}$(4)$, we get: \\
$\hspace*{2.00cm}$ $(x^*\odot y)\Cap (x\odot y^*)=(x^*\ra y^*)^*\Cap (x\ra y)^*=(y\ra x)^*\Cap (x\ra y)^*$ \\
$\hspace*{5.25cm}$ $=((y\ra x)\Cup (x\ra y))^*=1^*=0$. \\
$(2)$ By distributivity of $\odot$ over $\Cap$ and using $(1)$, we have: \\
$\hspace*{2.00cm}$ $(x\Cap x^*)\odot (y\Cap y^*)=((x\Cap x^*)\odot y)\Cap ((x\Cap x^*)\odot y^*)$ \\
$\hspace*{5.25cm}$ $=(x\odot y)\Cap (x^*\odot y)\Cap (x\odot y^*)\Cap (x^*\odot y^*)$ \\
$\hspace*{5.25cm}$ $=(x\odot y)\Cap 0\Cap (x^*\odot y^*)=0$. \\
$(3)$ Using $(2)$, we get: \\
$\hspace*{1.00cm}$ $(x\Cap x^*)\ra (y\Cup y^*)=((x\Cap x^*)\odot (y\Cup y^*)^*)^*
                    =((x\Cap x^*)\odot (y\Cap y^*))^*=0^*=1$. \\
Since $\le_Q$ and $\le$ coincide in $\mathcal{Z}(X)$, it follows that $x\Cap x^*\le_Q y\Cup y^*$. 
\end{proof}

\begin{definition} \label{lqbe-60} 
\emph{
A \emph{Kleene algebra} is a structure $(L,\wedge,\vee,{}^*,0,1)$, where $(L,\wedge,\vee,0,1)$ is a bounded 
distributive lattice and ${}^*$ is a unary operation satisfying the following conditions for all $x,y\in L$: \\
$(K_1)$ $(x^*)^*=x;$ \\
$(K_2)$ $(x\vee y)^*=x^*\wedge y^*;$ \\
$(K_3)$ $x\wedge x^*\le y\vee y^*$. 
}
\end{definition}

\begin{theorem} \label{lqbe-70} $(\mathcal{Z}(X),\Cap,\Cup,{}^*,0,1)$ is a Kleene algebra. 
\end{theorem}
\begin{proof}
It follows from Theorem \ref{lqbe-20-10} and Proposition \ref{lqbe-50}$(3)$. 
\end{proof}

$\vspace*{2mm}$

\section{The OML-center of quantum-Wajsberg algebras}

Given a quantum-Wajsberg algebra $X$, we define the OML-center $\mathcal{O}(X)$ of $X$, we study its properties, and  
show that $\mathcal{O}(X)$ is a subalgebra of $X$. 
We prove that $\mathcal{O}(X)$ is an orthomodular lattice, and the orthomodular lattices form a subvariety of the variety of quantum-Wajsberg algebras. 
In what follows, $(X,\ra,^*,1)$ will be a quantum-Wajsberg algebra, unless otherwise stated. 


\noindent 
Denote $\mathcal{O}(X)=\{x\in X\mid x=x^*\ra x\}$. 
Obviously $0,1\in \mathcal{O}(X)$. 

\begin{lemma} \label{oqbe-20} $\mathcal{O}(X)$ is closed under ${}^*$ and $\ra$. 
\end{lemma}
\begin{proof}
If $x\in \mathcal{O}(X)$, then $x=x^*\ra x$, and by Proposition \ref{qbe-80}$(8)$, we get 
$x^*=x\ra x^*=(x^*)^*\ra x^*$, hence $x^*\in \mathcal{O}(X)$. 
Let $x,y\in \mathcal{O}(X)$, that is $x=x^*\ra x$ and $y=y^*\ra y$. 
By Lemma \ref{qbe-10}$(9)$, we have $(x^*\ra y)^*\ra (x^*\ra y)=(x^*\ra x)^*\ra (y^*\ra y)=x^*\ra y$, 
thus $x^*\ra y\in \mathcal{O}(X)$. 
Finally, from $x^*,y\in \mathcal{O}(X)$, we get $x\ra y\in \mathcal{O}(X)$. 
Hence $\mathcal{O}(X)$ is closed under ${}^*$ and $\ra$.
\end{proof}

\begin{corollary} \label{oqbe-30} The following hold: \\
$(1)$ $\mathcal{O}(X)=\{x\in X\mid x^*=x\ra x^*\};$ \\
$(2)$ $(\mathcal{O}(X),\ra,0,1)$ is a subalgebra of $(X,\ra,0,1);$ \\
$(3)$ $\mathcal{O}(X)$ is closed under $\Cap$, $\Cup$ and $\odot$. 
\end{corollary}

\begin{proposition} \label{oqbe-40} $\mathcal{O}(X)=\{x\in X\mid x^*\Cup x=1\}=\{x\in X\mid x^*\Cap x=0\}$.  
\end{proposition}
\begin{proof}
If $x\in \mathcal{O}(X)$, then $x=x^*\ra x$, so that $x^*\Cup x=(x^*\ra x)\ra x=x\ra x=1$. 
Conversely, if $x^*\Cup x=1$, then $(x^*\ra x)\ra x=1$, that is $x^*\ra x\le x$. 
Since by Proposition \ref{qbe-60}$(2)$, $x\le_Q x^*\ra x$, using Proposition \ref{qbe-80}$(4)$ we get $x=x^*\ra x$, 
that is $x\in \mathcal{O}(X)$. Similarly $\mathcal{O}(X)=\{x\in X\mid x^*\Cap x=0\}$.
\end{proof}

\begin{proposition} \label{oqbe-40-10} The following hold for all $x\in \mathcal{O}(X)$ and $y\in X$: \\  
$(1)$ $x\ra (x\ra y)=x\ra y;$ \\
$(2)$ $(x\ra y)\ra x=x;$ \\
$(3)$ $(y\ra x)^*\ra x=y\ra x;$ \\
$(4)$ $(y\ra x)^*\ra (y\ra x)=y\ra (y\ra x)$.  
\end{proposition}
\begin{proof}
$(1)$ Using Lemma \ref{qbe-10}$(7)$, we get: $x\ra (x\ra y)=(x\ra y)^*\ra x^*=x\ra (y^*\ra x^*)=
y^*\ra (x\ra x^*)=y^*\ra x^*=x\ra y$. \\
$(2)$ It follows by $(1)$, applying Proposition \ref{qbe-80}$(8)$. \\
$(3)$ By Lemma \ref{qbe-10}$(7)$, $(y\ra x)^*\ra x=y\ra (x^*\ra x)=y\ra x$. \\
$(4)$ Replacing $y$ by $y^*$ in Lemma \ref{qbe-10}$(9)$ and taking into consideration that $x^*\ra x=x$, we get 
$(x^*\ra y^*)^*\ra (x^*\ra y^*)=x^*\ra (y\ra y^*)$, so that 
$(y\ra x)^*\ra (y\ra x)=y\ra (x^*\ra y^*)$. Hence $(y\ra x)^*\ra (y\ra x)=y\ra (y\ra x)$. 
\end{proof}

For any $x,y\in \mathcal{O}(X)$, define the operations: $x\Cup_L y=x^*\ra y$, $x\Cap_L y=x\odot y$ and the relation 
$x\le_L y$ iff $x^*\ra y=y$. 
One can easily check that $x\Cup_L y=(x^*\Cap_L y^*)^*$ and $x\Cap_L y=(x^*\Cup_L y^*)^*$. 

\begin{proposition} \label{oqbe-50} The following hold for all $x,y\in \mathcal{O}(X)$: \\
$(1)$ $\le_L={\le_Q}_{\mid \mathcal{O}(X)};$ \\
$(2)$ $x\Cup y\le_Q x\Cup_L y$ and $x\Cap_L y\le_Q x\Cap y;$ \\
$(3)$ $(x\Cup_L y)\ra x^*=x^*$ and $(x\Cap_L y)^*\ra x=x;$ \\
$(4)$ $(x\Cup_L y)^*\ra y=x\Cup_L y$ and $(x\Cap_L y)\ra y^*=(x\Cap_L y)^*;$ \\
$(5)$ $(\mathcal{O}(X),\Cap_L,\Cup_L,0,1)$ is a bounded lattice.    
\end{proposition}
\begin{proof}
$(1)$ Let $x,y\in \mathcal{O}(X)$ such that $x\le_Q y$. 
It follows that $y^*\le_Q x^*$ and $x^*\ra y\le_Q y^*\ra y=y$. 
On the other hand, $y\le_Q x^*\ra y$, hence $x^*\ra y=y$, that is $x\le_L y$. 
Conversely, if $x\le_L y$ we have $x\le_Q x^*\ra y=y$. 
Thus $\le_L={\le_Q}_{\mid \mathcal{O}(X)}$. \\ 
$(2)$ Since $x^*\le_Q x\ra y$, we have $(x\ra y)\ra y\le_Q x^*\ra y$, that is $x\Cup y\le_Q x\Cup_L y$. 
Similarly $x\le_Q x^*\ra y^*$ implies $(x^*\ra y^*)\ra y^*\le_Q x\ra y^*$.  
Hence $(x\ra y^*)^*\le_Q x\Cap y$, so that $x\odot y\le_Q x\Cap y$, that is $x\Cap_L y\le_Q x\Cap y$. \\
$(3)$ It follows from Proposition \ref{oqbe-40-10}$(2)$, replacing $x$ by $x^*$ and $y$ by $y^*$, respectively. \\
$(4)$ Since $y\in \mathcal{O}(X)$, by Proposition \ref{oqbe-40-10}$(3)$ we have $(x\ra y)^*\ra y=x\ra y$. 
Replacing $x$ by $x^*$ we get $(x\Cup_L y)^*\ra y=x\Cup_L y$, and replacing $y$ by $y^*$ we have 
$(x\Cap_L y)\ra y^*=(x\Cap_L y)^*$. \\
$(5)$ Clearly $\Cup_L$ and $\Cap_L$ are commutative and idempotent. 
Moreover, using Lemma \ref{qbe-10}$(7)$ we can easily check that $\Cup_L$ and $\Cap_L$ are associative. 
Finally, applying Proposition \ref{oqbe-40-10}$(2)$, we have: \\
$\hspace*{2cm}$ $x\Cup_L (x\Cap_L y)=x^*\ra (x\ra y^*)^*=(x\ra y^*)\ra x=x$, \\
$\hspace*{2cm}$ $x\Cap_L (x\Cup_L y)=(x\ra (x^*\ra y)^*)^*=((x^*\ra y)\ra x^*)^*=(x^*)^*=x$, \\
for all $x,y\in \mathcal{O}(X)$, hence $\Cup_L$ and $\Cap_L$ satisfy the absorption laws. 
Thus $(\mathcal{O}(X),\Cap_L,\Cup_L,0,1)$ is a bounded lattice. 
\end{proof}

\begin{corollary} \label{oqbe-60} The following hold for all $x,y\in \mathcal{O}(X)$: \\
$(1)$ $x\le_Q y$ iff $y=y\Cup_L x;$ \\
$(2)$ $x\Cup y=(x\ra y)^*\Cup_L y$. 
\end{corollary}
\begin{proof} 
$(1)$ $x\le_Q y$ iff $x\le_L y$ iff $y=x^*\ra y=y^*\ra x=y\Cup_L x$. \\
$(2)$ $x\Cup y=(x\ra y)\ra y=((x\ra y)^*)^*\ra y=(x\ra y)^*\Cup_L y$. 
\end{proof}
\noindent
In what follows, if $x,y\in \mathcal{O}(X)$, we will use $x\le_Q y$ instead of $x\le_L y$. 

\begin{proposition} \label{oqbe-70} For any $x,y\in \mathcal{O}(X)$, $x\Cup_L y$ and $x\Cap_L y$ are the l.u.b. and 
g.l.b. of $\{x,y\}$, respectively. 
\end{proposition}
\begin{proof}
Obviously $x,y\le_Q x^*\ra y$, so that $x\Cup_L y$ is an upper bound of $\{x,y\}$.  
Let $z\in \mathcal{O}(X)$ be another upper bound of $\{x,y\}$ in $\mathcal{O}(X)$, that is $x,y\le_Q z$. 
It follows that $z^*\le_Q x^*$, so that $x^*\ra y\le_Q x^*\ra z\le_Q z^*\ra z=z$. 
Hence $x\Cup_L y\le_Q z$, that is $x\Cup_L y$ is the l.u.b. of $\{x,y\}$. 
Similarly $x\odot y\le_Q x,y$, thus $x\Cap_L y$ is a lower bound of $\{x,y\}$. 
Let $z\in \mathcal{O}(X)$ be another lower bound of $\{x,y\}$ in $\mathcal{O}(X)$, so that $z\le_Q x,y$. 
We get $x^*,y^*\le_Q z^*$, so that $z^*$ is an upper bound of $\{x^*,y^*\}$, hence $x^*\Cup_L y^*\le_Q z^*$, 
that is $x\ra y^*\le_Q z^*$. 
It follows that $z\le_Q (x\ra y^*)^*=x\odot y=x\Cap_L y$, thus $x\Cap_L y$ is the g.l.b. of $\{x,y\}$. 
\end{proof}

\begin{definition} \label{oqbe-80} $\rm($\cite{Burris}$\rm)$ 
\emph{
An algebra $(X,\wedge,\vee,{}^{'},0,1)$ with two binary, one unary and two nullary operations is an 
\emph{ortholattice} if it satisfies the following axioms for all $x,y,z\in X$: \\
$(Q_1)$ $(X,\wedge,\vee,0,1)$ is a bounded lattice; \\
$(Q_2)$ $x\wedge x^{'}=0$ and $x\vee x^{'}=1;$ \\
$(Q_3)$ $(x\wedge y)^{'}=x^{'}\vee y^{'}$ and $(x\vee y)^{'}=x^{'}\wedge y^{'};$ \\
$(Q_4)$ $(x^{'})^{'}=x$. \\
An \emph{orthomodular lattice} is an ortholattice satisfying the following axiom: \\
$(Q_5)$ $x\le y$ implies $x\vee (x^{'}\wedge y)=y$ (where $x\le y$ iff $x=x\wedge y$). 
}
\end{definition}

\begin{theorem} \label{oqbe-90} $(\mathcal{O}(X),\Cap_L,\Cup_L,{}^*,0,1)$ \emph{is an orthomodular lattice called 
the \emph{orthomodular center} or \emph{OML-center}} of $X$. 
\end{theorem}
\begin{proof}
Let $X$ be a quantum-Wajsberg algebra.
Using Propositions \ref{oqbe-50}, \ref{oqbe-70}, \ref{oqbe-40} we can easily check that $(\mathcal{O}(X),\Cap_L,\Cup_L,{}^*,0,1)$ is an ortholattice. We show that axiom $(Q_5)$ is also satisfied. 
Let $x,y\in \mathcal{O}(X)$ such that $x\le_Q y$, and we have: 
$x\Cup_L (x^*\Cap_L y)=x\Cup_L (x^*\odot y)=x\Cup_L (x^*\ra y^*)^*=x\Cup_L (y\ra x)^*=x^*\ra (y\ra x)^*=
(y\ra x)\ra x=y\Cup x=y$, since $x\le_Q y$. 
\end{proof}

\begin{theorem} \label{oqbe-90-10} If $(X,\wedge,\vee,{}^{'},0,1)$ is an orthomodular lattice, then $(X,\ra,0,1)$ is 
a quantum-Wajsberg algebra, where $x\ra y=x^{'}\vee y$ for all $x,y\in X$. 
\end{theorem}
\begin{proof}
According to \cite[Thm. 2.3.9]{DvPu}, every orthomodular lattice $(X,\wedge,\vee,{}^{'},0,1)$ determines a QMV 
algebra by taking $\oplus$ as the supremum $\vee$ and ${}^*$ as the orthocomplement ${}^{'}$, and conversely, 
if an ortholattice $X$ determines a QMV algebra $(X,\oplus,^*,0,1)$ taking $\oplus=\vee$ and ${}^*={}^{'}$, 
then $X$ is orthomodular. 
By \cite[Thm. 5.3]{Ciu78}, any quantum-MV algebra $(X,\oplus,^*,0,1)$ is a quantum-Wajsberg algebra $(X,\ra,0,1)$, 
where $x\ra y=x^*\oplus y$. 
It follows that every orthomodular lattice $(X,\wedge,\vee,{}^{'},0,1)$ determines a quantum-Wajsberg algebra $(X,\ra,0,1)$ with $x\ra y=x^*\oplus y=x^*\vee y$ for all $x,y\in X$. 
\end{proof}

\begin{corollary} \label{oqbe-100} $(X,\Cap_L,\Cup_L,{}^*,0,1)$ is an orthomodular lattice if and only if $\mathcal{O}(X)=X$.  
\end{corollary}
\noindent
Similarly as \cite[Cor. 2.3.13]{DvPu} for the case of QMV algebras, we have the following result. 

\begin{corollary} \label{oqbe-110} The orthomodular lattices form a subvariety of the variety of quantum-Wajsberg 
algebras. This subvariety satisfies the condition $x=x^*\ra x$, or equivalently, $x^*\Cup x=1$, 
or equivalently, $x^*\Cap x=0$.  
\end{corollary}
\begin{proof}
The equivalence of conditions $x=x^*\ra x$, $x^*\Cup x=1$ and $x^*\Cap x=0$ follows from Proposition \ref{oqbe-40}. 
If a quantum-Wajsberg algebra $(X,\ra,0,1)$ is an orthomodular lattice with $x\vee y=x^*\ra y$, than 
$x^*\ra x=x\vee x=x$. 
Conversely, if $X$ satisfies condition $x^*\ra x=x$ for any $x\in X$, then $\mathcal{O}(X)=X$, hence $X$ is an 
orthomodular lattice. 
\end{proof}

\begin{example} \label{oqbe-120} 
Let $X=\{0,a,b,c,d,1\}$ and let $(X,\ra,0,1)$ be the involutive BE algebra with $\ra$ and the corresponding 
operation $\Cap$ given in the following tables:  
\[
\begin{array}{c|ccccccc}
\ra & 0 & a & b & c & d & 1 \\ \hline
0   & 1 & 1 & 1 & 1 & 1 & 1 \\ 
a   & c & 1 & 1 & c & 1 & 1 \\ 
b   & d & 1 & 1 & 1 & d & 1 \\ 
c   & a & a & 1 & 1 & 1 & 1 \\
d   & b & 1 & b & 1 & 1 & 1 \\
1   & 0 & a & b & c & d & 1
\end{array}
\hspace{10mm}
\begin{array}{c|ccccccc}
\Cap & 0 & a & b & c & d & 1 \\ \hline
0    & 0 & 0 & 0 & 0 & 0 & 0 \\ 
a    & 0 & a & b & 0 & d & a \\ 
b    & 0 & a & b & c & 0 & b \\ 
c    & 0 & 0 & b & c & d & c \\
d    & 0 & a & 0 & c & d & d \\
1    & 0 & a & b & c & d & 1
\end{array}
.
\]
Then $X$ is a quantum-Wajsberg algebra and $\mathcal{Z}(X)=\{0,1\}$, $\mathcal{O}(X)=X$. 
Therefore $(X,\Cap_L,\Cup_L,{}^*,0,1)$ is an orthomodular lattice with $\Cup_L$ and $\Cap_L$ given below. 
\[
\begin{array}{c|ccccccc}
\Cup_L & 0 & a & b & c & d & 1 \\ \hline
0      & 0 & a & b & c & d & 1 \\ 
a      & a & a & 1 & 1 & 1 & 1 \\ 
b      & b & 1 & b & 1 & 1 & 1 \\ 
c      & c & 1 & 1 & c & 1 & 1 \\
d      & d & 1 & 1 & 1 & d & 1 \\
1      & 1 & 1 & 1 & 1 & 1 & 1
\end{array}
\hspace{10mm}
\begin{array}{c|ccccccc}
\Cap_L & 0 & a & b & c & d & 1 \\ \hline
0      & 0 & 0 & 0 & 0 & 0 & 0 \\ 
a      & 0 & a & 0 & 0 & 0 & a \\ 
b      & 0 & 0 & b & 0 & 0 & b \\ 
c      & 0 & 0 & 0 & c & 0 & c \\
d      & 0 & 0 & 0 & 0 & d & d \\
1      & 0 & a & b & c & d & 1
\end{array}
.
\]
As we can see in this example, in general, $\Cup_L\neq \Cup$ and $\Cap_L\neq \Cap$. 
\end{example} 

\begin{remark} \label{oqbe-130} In general, the lattice $(\mathcal{O}(X),\Cap_L,\Cup_L,0,1)$ is not distributive. 
Indeed, in Example \ref{oqbe-120} we have $a\Cup_L (b\Cap_L c)=a\neq 1=(a\Cup_L b)\Cap_L (a\Cup_L c)$.  
\end{remark}


$\vspace*{5mm}$

\section{Concluding remarks and future work}

In this paper, we continued the study of quantum-Wajsberg algebras (\cite{Ciu78}). 
We defined the  Wajsberg-center and the OML-center of a quantum-Wajsberg algebra $(X,\ra,^*,1)$, 
proving that the Wajsberg-center is a Wajsberg subalgebra of $X$, and that it is a distributive sublattice of the 
poset $(X,\le_Q,0,1)$ (where $0=1^*$). 
We introduced the notion of quasi-linear quantum-Wajsberg algebras, and we proved that the Wajsberg-center of a quasi-linear quantum-Wajsberg algebra is a linearly ordered Wajsberg algebra.
We also proved that the OML-center is an orthomodular lattice, and that the orthomodular lattices form a subvariety 
of the variety of quantum-Wajsberg algebras. \\ 
There are several ways this work can be continued, as follows: \\
$\hspace*{0.5cm}$ $-$ Introduce and study certain generalizations of quantum-Wajsberg algebras, 
such as implicative-orthomodular, pre-Wajsberg and meta-Wajsberg algebras. \\
$\hspace*{0.5cm}$ $-$ Define the implicative-orthomodular lattices as a special subclass of quantum-Wajsberg algebras, and study their properties. \\
$\hspace*{0.5cm}$ $-$ Prove an analogue of Foulis-Holland theorem for implicative-orthomodular lattices. \\
$\hspace*{0.5cm}$ $-$ Study the Baer $^*$-semigroup associated to an implicative-orthomodular lattice $X$ and its relationship with the Sasaki projections defined on $X$. \\
$\hspace*{0.5cm}$ $-$ Investigate the central lifting property for implicative-orthomodular lattices. \\
Another direction of research could be the solving of the following open problem. \\ 
{\bf Open problem.} Is the variety of quasi-linear quantum-Wajsberg algebras axiomatizable (in the sense 
of \cite{Giunt3})?

$\vspace*{1mm}$

\section* {\bf\leftline {Declaration of interest}}

\noindent The author has no relevant financial or non-financial interests to disclose. \\

$\vspace*{1mm}$
          
\begin{center}
\sc Acknowledgement 
\end{center}
The author is very grateful to the anonymous referees for their useful remarks and suggestions on the subject that helped improving the presentation.

$\vspace*{1mm}$

\end{document}